\documentclass[12pt]{amsart}

\usepackage{amsmath,amssymb,amsthm,latexsym}
\usepackage[dvips]{graphicx}
\usepackage[T1]{fontenc}
\usepackage[centertags]{amsmath}
\usepackage{amsfonts}
\usepackage{amscd}
\usepackage{euscript}
\usepackage{graphicx}
\usepackage{color}
\usepackage[all]{xy}
\usepackage[T1]{fontenc}
\usepackage{parskip}
\usepackage{enumitem}

\newcommand{\B}[1]{{\mathbf #1}}

\makeatletter
\def\@settitle{\begin{center}%
  \baselineskip14\p@\relax
    {\Large\textit \@title}
  \end{center}%
}
\def\@setauthors{%
  \begingroup
  \def\thanks{\protect\thanks@warning}%
  \trivlist
  \centering\footnotesize \@topsep30\p@\relax
  \advance\@topsep by -\baselineskip
  \item\relax
  \author@andify\authors
  \def\\{\protect\linebreak}%
  {\scshape \authors}%
  \ifx\@empty\contribs
  \else
    ,\penalty-3 \space \@setcontribs
    \@closetoccontribs
  \fi
  \endtrivlist
  \endgroup
}
\makeatother

\newtheorem{theorem}[subsection]{Theorem}
\newtheorem{corollary}[subsection]{Corollary}
\newtheorem{lemma}[subsection]{Lemma}
\newtheorem{proposition}[subsection]{Proposition}
\theoremstyle{definition}

\newtheorem{example}[subsection]{Example}

\theoremstyle{remark}

\newtheorem*{remark*}{Remark}

\numberwithin{equation}{section}
\numberwithin{figure}{section}
\numberwithin{table}{section}

\newcommand{\OP}{\operatorname}

\definecolor{dred}{RGB}{200,00,10}

\begin{document}

\title[Concordance group]{Concordance group and stable commutator length in braid groups}
\author[Brandenbursky]{Michael Brandenbursky}
\address{CRM, University of Montreal, Canada}
\email{michael.brandenbursky@mcgill.ca}
\author[K\k{e}dra]{Jarek K\k{e}dra}
\address{University of Aberdeen and University of Szczecin}
\email{kedra@abdn.ac.uk}

\begin{abstract}
We define  quasihomomorphisms from braid groups to the concordance group of
knots and examine its properties and consequences of its existence. In
particular, we provide a relation between the stable four ball genus in the
concordance group and the stable commutator length in braid groups, and produce
examples of infinite families of concordance classes of knots with uniformly
bounded four ball genus.  We also provide applications to the geometry of the
infinite braid group $\B B_\infty$.  In particular, we show that the commutator
subgroup $[\B B_{\infty},\B B_{\infty}]$ admits a stably unbounded conjugation
invariant norm. This answers an open problem posed by Burago, Ivanov and
Polterovich.
\end{abstract}

\maketitle
\section{Introduction} \label{S:intro}

We define a map $\Psi_n\colon \B B_n\to \OP{Conc}(\B S^3)$ from the braid group
on $n$ strings to the concordance group of knots in the three dimensional
sphere and observe that it has good algebraic and geometric properties. The map
is defined by closing braids appropriately. The maps defined for various
numbers of strings  are compatible with inclusions and hence they induce a map
on the infinite braid group. We examine the geometric properties with
respect to the conjugation invariant word norm on the braid group and the norm
defined by the four ball genus on the concordance group. The first observation
is that the map $\Psi_n$ is a quasihomomorphism with the defect depending on
$n$ and the second is that the map is Lipschitz with the Lipschitz constant
independent on~$n$. The latter implies that the map defined on the infinite
braid group is also Lipschitz.

We then investigate the consequences of the above observations. Using our
knowledge of the conjugation invariant geometry of braid groups we provide
applications to knot theory. Namely we give recipes for producing
infinite families of knots which have uniformly bounded four ball genus.
Such examples were known before so what is new here is our method. We
also produce examples of infinite families of prime knots with unbounded
stable four ball genus. We relate the stable commutator length
on braid groups to the stable four ball genus which answers a question
of Livingston \cite{MR2745668} (who, in turn, attributes the question to Calegari).

Our main application to the geometry of braids is to prove that the commutator
subgroup of the infinite braid group is stably unbounded.  The latter means
that there is an element $g\in [\B B_{\infty},\B B_{\infty}]$ such that
$\lim_{n\to \infty}\frac{\|g^n\|}{n}>0$, where $\|g\|$ denotes the conjugation
invariant word norm of $g$.  This example is interesting because this (perfect)
group admits no nontrivial quasimorphisms (the main tool for proving stable
unboundedness) and its commutator length is bounded.  This answers an open
problem from the paper by Burago, Ivanov and Polterovich \cite{MR2509711}.  Another example for
the same problem has been recently provided by Kawasaki \cite{Kawasaki} who showed
that the group of symplectic diffeomorphisms of the Euclidean space is stably
unbounded.

In the remaining part of the introduction we state our results and provide
more details.

\subsection*{Two elementary observations}
Let $A$ be an abelian group equipped with a pseudo\-norm and
let $|a|$ denote the value of the pseudo\-norm on an element
$a\in~A$. Let $G$ be a group. A map
$\psi\colon G\to A$ is called a {\em quasihomomorphism} if there
exist a constant $D_{\psi}\geq 0$ such that
$$
|\psi(g)-\psi(gh)+\psi(h)|\leq D_{\psi}
$$
for all $g,h\in G$. The number $D_{\psi}$ is called the defect of $\psi$.
A real valued quasihomomorphism is traditionally called
a {\em quasimorphism}.

Let $(\phantom{a})\colon \B B_n\to \Sigma_n$ be the natural
projection from the braid group to the symmetric group.
Let $\B B_n^K$ denote the subset consisting of braids whose closures
are knots. We define a projection
$\pi_n\colon \B B_n\to \B B_n^K$
by sending a braid $\alpha$ to the braid
$\alpha\sigma_{(\alpha)}$, where the braid $\sigma_{(\alpha)}$
depends only on the permutation induced by $\alpha$
(see Section \ref{SS:Psin}).

Let $\OP{Conc}(\B S^3)$ denote the group of smooth concordance classes of
knots in the three dimensional sphere. It is equipped with
a norm defined by the four ball genus
(see Section \ref{S:preliminaries} for details).
Let $\Psi_n\colon\B B_n\to \OP{Conc}(\B S^3)$ be
defined as the composition
$$
\B B_n\stackrel{\pi_n}\longrightarrow
\B B_n^K \stackrel{\OP{closure}}\longrightarrow
{\bf KNOTS}\stackrel{[-]}\longrightarrow \OP{Conc}(\B S^3)
$$
$$
\Psi_n(\alpha) := [\widehat{\alpha\sigma_{(\alpha)}}],
$$
where $\widehat{\alpha}$ denotes the closure of the braid
$\alpha$. Here is our first observation which was essentially
proven by the first author in \cite{MR2851716}.

\begin{theorem}\label{T:main}
The map  $\Psi_n\colon \B B_n\to \OP{Conc}(\B S^3)$ is a quasihomomorphism with
respect to the four ball genus norm and with
defect  $D_{\Psi_n}\leq 3n+1$. Its image contains all concordance
classes represented by knots which are closures of braids on
$n$ strings.
\end{theorem}

\begin{remark*}
The quasihomomorphisms $\Psi_n$ are compatible with
the inclusions $\B B_n\to \B B_{n+1}$ and hence
the above construction defines a  surjective
map $\Psi_{\infty}\colon\B B_{\infty}\to \OP{Conc}(\B S^3)$
(see Section \ref{SS:Psin}).
However, using the fact that every homogeneous quasimorphism on $\B B_{\infty}$
must be a homomorphism \cite{MR2509718}, we show that the map
$\Psi_{\infty}$ can't be a quasihomomorphism~(Proposition \ref{P:infty}).
\end{remark*}

{\em Convention.} It is always assumed that $n$ in the notation
$\B B_n$ for the braid group is a natural number. Some of our
statements extend to the infinite braid group. In such cases
we emphasize that $n\in \B N\cup \{\infty\}$.


Let $\sigma_1,\ldots,\sigma_{n-1}\in \B B_n$ be the
standard Artin generators of the braid group. That
is, the braid $\sigma_i$ swaps the $i$-th and the
$(i+1)$-st string. Observe that these braids are
pairwise conjugate thus $\B B_{n}$ is normally generated
by the symmetric set $\{\sigma_1^{\pm 1}\}$,
where $n\in \B N\cup\{\infty\}$.
Let us consider the associated conjugation invariant word norm
on $\B B_n$, denoted by $\|\alpha\|$, and
the induced biinvariant metric
defined by $\OP{d}(\alpha,\beta):=\|\alpha\beta^{-1}\|$.

\begin{theorem}\label{T:lipschitz}
Let $n\in\B N\cup\{\infty\}$.
The map $\Psi_n\colon \B B_n\to \OP{Conc}(\B S^3)$
is Lipschitz with respect to the biinvariant word norm on
the braid group and the four ball genus norm on the concordance
group. More precisely,
$$
\OP{g_4}(\Psi_n(\alpha))\leq \frac{1}{2}\|\alpha\|
$$
for all braids $\alpha\in \B B_n$.
\end{theorem}

\begin{remark*}
\begin{enumerate}[leftmargin=*]
\item
It follows from Theorem \ref{T:main} that $\Psi_n$ is Lipschitz
with constant bounded above by the defect $D_{\Psi_n}$. We get
a smaller constant in the above theorem by a more direct
and elementary geometric argument.
\item
If one defines a metric on the concordance group by
$$
\OP{d_4}(K,L):=\OP{g_4}(K-L)
$$
then it follows from the above theorems that
$$
\OP{d_4}(\Psi_n(\alpha),\Psi_n(\beta))
\leq \frac{1}{2}\OP{d}(\alpha,\beta) + D_{\Psi_n}.
$$
That is, the map $\Psi_n$ is large scale Lipschitz with respect
to the metrics and for any natural number $n\in \B N$.
\end{enumerate}
\end{remark*}

In the remaining part of this introduction we discuss applications
and consequences of the above observations.

\subsection*{Quasimorphisms on braid groups}
Composing the quasihomomorphism $\Psi_n$ with a suitable quasimorphism
defined on the concordance group yields a quasimorphism
on the braid group. More precisely, we have the following
observation.

\begin{corollary}\label{C:compositions}
Let $\varphi\colon \OP{Conc}(\B S^3)\to \B R$ be a quasimorphism.
If $\varphi$ is
Lipschitz with respect to the four ball genus norm then the composition
$\varphi\circ \Psi_n\colon \B B_n\to \B R$ is a quasimorphism.
\end{corollary}

This idea was used by the first author in~\cite{MR2851716}.
The next applications provide new results.

\subsection*{The quasihomomorphism $\Psi_n$ is Lipschitz with respect to the commutator length}
The commutator length $\OP{cl}(g)$ of an element $g$ in $[G,G]$ is defined
to be the minimal number of commutators in $G$ whose product
is equal to $g$. The following result  is an
application of Theorem \ref{T:main} and a general fact
about quasihomomorphisms presented in Lemma \ref{L:qm}.

\begin{corollary}\label{C:lipschitz-bi}
The restriction of the quasihomomorphism $\Psi_n$ to the
commutator subgroup $[\B B_n,\B B_n]$ is Lipschitz with
respect to the commutator length. More precisely,
$$
\OP{g_4}(\Psi_n(\alpha))\leq 6D_{\Psi_n}\, \OP{cl}(\alpha)
$$
for any $\alpha\in [\B B_n,\B B_n]$.
\end{corollary}

The above result does not extend to the infinite case
because the commutator length is bounded by $2$ on the
infinite braid group, according to Burago, Ivanov and Polterovich
(see Theorem \ref{T:bip}).

\subsection*{A relation between the stable four ball genus and the scl}
Livingston asked in \cite[Section 8.1]{MR2745668} whether there
is a connection between the stable commutator length in groups and the
stable four ball genus in $\OP{Conc}(\B S^3)$.  The next corollary provides
such a connection.

\begin{corollary}\label{C:scl}
If $\alpha \in [\B B_n,\B B_n]$ then
$$
\OP{sg_4}(\Psi_n(\alpha)) \leq 6D_{\Psi_n}\OP{scl}(\alpha) + D_{\Psi_n}.
$$
In particular, if the stable commutator length
of $\alpha$ is trivial then the stable four ball genus of
$\Psi_n(\alpha)$ is bounded above by the defect $D_{\Psi_n}$:
$$
\OP{scl}(\alpha) = 0 \qquad \Longrightarrow \qquad \OP{sg_4}(\Psi_n(\alpha))
\leq D_{\Psi_n}.
$$
\end{corollary}

\begin{remark*}
The braids $\alpha^{2n}$
from Example \ref{E:s1s3}
have trivial stable commutator length
and $\OP{sg_4}(\Psi_4(\alpha^{2n}))>0$.
The last inequality follows from the fact that the
$\Psi_4(\alpha^{2n})=[T_{2n+1}\# T_{2n-1}^*]$, where
$T_k$ is the torus knot obtained by closing the braid
$\sigma_1\in \B B_2$ and $T_k^*$ is its mirror image. The signature of $T_{2n+1}\# T_{2n-1}^*$
is equal to two and hence its stable four ball genus is
bounded from below by one, due to Murasugi inequality \eqref{Eq:Murasugi}.
\end{remark*}

\subsection*{Families of knots with uniformly bounded four ball genus}
The next result
can be used to produce  concrete infinite families of knots
(and concordance classes) with uniformly bounded four ball genus.

\begin{corollary}\label{C:bounded}
Let $\alpha \in [\B B_n,\B B_n]$.
If $\OP{scl}(\alpha)=0$ then the concordance classes
$\Psi_n(\alpha^k)$, for $k\in \B Z$,
have uniformly bounded four ball genus.
\end{corollary}

\begin{remark*}
Infinite families of knots with bounded four ball genus
have been known since the 1960's \cite[Section 3.1]{MR2179265}.
Since it is easy to provide braids with trivial stable commutator length,
our corollary yields many families of knots for which checking the
boundedness of the four ball genus could be difficult otherwise.
\end{remark*}

Examples of braids with trivial stable commutator length abound.
For instance, a braid which is conjugate to its inverse has
trivial stable commutator length.

\begin{example}\label{E:s1S2}
Let $\alpha = \sigma_1\sigma_2^{-1}\in \B B_3$. It is
straightforward to see that $\Delta\alpha\Delta^{-1}=\alpha^{-1}$,
where $\Delta=\sigma_1\sigma_2\sigma_1$ is the Garside element.
Consequently, $\OP{scl}(\alpha)=0$ and it follows from Corollary
\ref{C:bounded} that the family consisting of the closures of
the braids $\alpha^k\sigma_{(\alpha^k)}$ has uniformly bounded
four ball genus. It is not difficult to show that this
family is infinite. However, it remains an open question
whether the family $\Psi_3(\alpha^k)$ of concordance classes
is infinite. It is known that each concordance class $\Psi_3(\alpha^k)$
is of order at most two in $\OP{Conc}(\B S^3)$.
\end{example}

\begin{example}\label{E:s1S3}
Let $\alpha=\sigma_1\sigma_3^{-1}\in \B B_4$. It is again straightforward
to see that this braid is conjugate to its inverse and hence it has
trivial stable commutator length. However, in this case we obtain
that the set of concordance classes $\Psi(\alpha^{2n})$ is infinite
(Section \ref{SS:s1S3}) and has uniformly bounded four ball genus.
\end{example}

\subsection*{Families of prime knots with unbounded stable four ball genus}
Let $G$ be a group and $\varphi\colon G\to\B R$ be a quasimorphism. We denote by
$\overline{\varphi}\colon G\to\B R$ the homogenization of $\varphi$, i.e.
$\overline{\varphi}(g):=\lim_{p\to\infty}\varphi(g^p)/p$. For more information about
quasimorphisms see \cite{MR2527432}.

\begin{corollary}\label{C:unboundedsg}
Let $\varphi\colon \OP{Conc}(\B S^3)\to \B R$ be a quasimorphism
which is Lipschitz with respect to the four ball genus norm.
Let $C_{\varphi}$ denote its Lipschitz constant.
If $\alpha \in \B B_n$ and $p\in \B N$ then
$$
\OP{sg_4}(\Psi_n(\alpha^p))\geq
\frac{\left|(\overline{\varphi\circ\Psi_n})(\alpha)\right|}{C_{\varphi}}\cdot p - D_{\Psi_n}.
$$
where $\overline{\varphi\circ\Psi_n}$ denotes the homogenization
of the quasimorphism $\varphi\circ\Psi_n$.
If particular, if the quasimorphism $\varphi\circ\Psi_n$ is unbounded
on the cyclic subgroup generated by $\alpha$ then the stable
genus of the knots $\Psi_n(\alpha^p)$ grows linearly with $p$.
\end{corollary}

\begin{example}\label{E:signature}
Let $\varphi \colon \OP{Conc}(\B S^3)\to \B R$ be a quasimorphism
given by the signature of a knot. It is known \cite{MR0171275} that
it is Lipschitz with respect to the four ball genus and hence
we can apply the above corollary. In this example we show that there exists a
braid $\alpha\in\B B_3$, such that for each $p\in\B N$ the knot
$\Psi_3(\alpha^p)$ is prime and the composition $\varphi\circ \Psi_3$ is
unbounded on the cyclic subgroup generated by $\alpha$. The braid $\alpha$ is
given by the following presentation
$\alpha=\sigma_1^{-4}\sigma_2\sigma_1^2\sigma_2\in \B B_3$. The fact that
$\Psi_3(\alpha^p)$ is a prime knot for each $p\in\B N$ follows from
\cite{MR542687}.
\end{example}

\subsection*{Applications to the biinvariant geometry of
the infinite braid group}
Recall that a norm $\nu$ on a group $G$ is called {\em stably
unbounded} if there exists $g\in G$ such that
$$
s\nu(g)=\lim_{p\to \infty}\frac{\nu(g^p)}{p}\neq 0.
$$
If $\psi\colon G\to \B R$ is a nontrivial homogeneous
quasimorphism which is Lipschitz with respect to $\nu$
then $\nu$ is stably unbounded. This is the usual argument
proving the stable unboundedness of a norm.

It follows from a result of Kotschick \cite{MR2509718} that the
only nontrivial homogeneous quasimorphism on the
infinite braid group is the abelianisation (up to a constant).
Moreover, the commutator length on $[\B B_{\infty},\B B_{\infty}]$
is bounded by two according to Burago, Ivanov and Polterovich
\cite[Theorem 2.2]{MR2509711}.
More precisely, they proved the following result.

\begin{theorem}[Burago-Ivanov-Polterovich]\label{T:bip}
Let  $H$ be a subgroup of a group $G$.
Suppose that for every natural number $m\in \B N$
there exists an element $g\in G$
such that the conjugate subgroups $g^iHg^{-i}$ and
$g^jHg^{-j}$ pairwise commute for $0\leq i<j\leq m$.
Then the commutator length in $G$ of every element $h\in H$
is bounded by two: $\OP{cl}_G(h)\leq 2$.
\end{theorem}

Observe that the hypothesis of the above theorem
is satisfied by the braid groups
$\B B_n\subset \B B_{\infty}$ for
every $n\in \B N$ (see the proof of Proposition \ref{P:disp}
for a detailed argument).
This implies that the commutator length on the infinite braid group
is bounded by two.

On the other hand, the diameter of the infinite braid group
with respect to the biinvariant word metric is infinite.
To see this consider
the projection $\B B_{\infty}\to \Sigma_{\infty}$ to the
infinite symmetric group. It is Lipschitz and the cardinality
of the support of a permutation defines a conjugation invariant
norm on the symmetric group which is clearly unbounded. This implies that
the biinvariant word norm is unbounded on the infinite
braid group. The argument, however, says nothing on the
geometry of cyclic subgroups of the infinite braid group
and, in particular, it does not answer the question whether
the word norm is stably unbounded.
Our next corollary answers this question affirmatively.

\begin{corollary}\label{C:binfty}
Let $\alpha \in [\B B_{\infty},\B B_{\infty}]$. If there exists
a Lipschitz quasimorphism $\varphi\colon \OP{Conc}(\B S^3)\to \B R$
such that $\varphi(\Psi_{\infty}(\alpha))\neq 0$ then
$$
\lim_{p\to \infty}\frac{\|\alpha^p\|}{p}>0.
$$
In particular, the braid
$\sigma_1^{-4}\sigma_2\sigma_1^2\sigma_2\in [\B B_\infty,\B B_\infty]$
discussed in Example~\ref{E:signature} satisfies the above assumption
and hence the conjugation invariant word norm on $\B B_{\infty}$
is stably unbounded.
\end{corollary}

Burago, Ivanov and Polterovich posed several
problems about existence of groups with certain metric properties
\cite{MR2509711}.
One of them asks if there exists a group $G$ with the following
properties:
\begin{enumerate}[leftmargin=*]
\item
$G$ has finite abelianisation,
\item
the commutator length of $G$ is stably trivial,
\item
$G$ admits a stably unbounded conjugation invariant norm.
\end{enumerate}

The infinite braid group satisfies the last two conditions
of the above problem but its abelianisation is infinite cyclic.
We have, however, the following observation.

\begin{theorem}\label{E:bip}
The commutator subgroup $[\B B_{\infty},\B B_{\infty}]$ of
the infinite braid group satisfies the conditions of the
above problem.
\end{theorem}

\begin{proof}
Observe that the commutator subgroup $[\B B_{\infty},\B B_{\infty}]$
is the union of the commutator subgroups $[\B B_n,\B B_n]$ of
the braid group on finitely many strings.
Let us justify that the group $[\B B_{\infty},\B B_{\infty}]$
satisfies the properties of the above problem.
\begin{enumerate}[leftmargin=*]
\item
It is known \cite{MR0251712} that the commutator subgroup
$[\B B_n,\B B_n]$ of the braid group is perfect for $n>4$.
This implies that the group $[\B B_{\infty},\B B_{\infty}]$
is perfect as well. Equivalently, its abelianisation is
trivial.
\item
Observe that the subgroups
$[\B B_{n},\B B_{n}]\subset [\B B_{\infty},\B B_{\infty}]$ satisfy the
assumption of Theorem \ref{T:bip} (Proposition \ref{P:disp}).
This implies that the
commutator length is bounded by two and, in particular, it
is stably trivial.
\item
The restriction of the conjugation invariant word norm from the
whole infinite braid group to its commutator subgroup is
stably unbounded due to Corollary \ref{C:binfty}.
\end{enumerate}
\end{proof}

\subsection*{Acknowledgements}
We would like to thank Steve Boyer, Ana Garcia Lecuona, Paolo Lisca and Brendan Owens for
useful conversations, and Micha\l\ Marcinkowski for simplifying our proof
of Corollary \ref{C:scl}.

First author was partially supported by the CRM-ISM fellowship. He would like
to thank CRM-ISM Montreal for the support and great research atmosphere. Part
of this work has been done during the authors stay at Max Planck Institute for
Mathematics in Bonn. We wish to express our gratitude to the Institute for the
support and excellent working conditions.

\section{Preliminaries}\label{S:preliminaries}
\subsection{A norm on a group}
Let $\nu\colon G\to \B R$ be a function. It is called a {\em pseudonorm}
if it satisfies the following conditions for all $g,h\in G$:
\begin{enumerate}
\item
$\nu(g)\geq 0$
\item
$\nu(g)=\nu(g^{-1})$
\item
$\nu(gh)\leq \nu(g) + \nu(h)$
\end{enumerate}
If, in addition, $\nu(g)=0$ if and only if $g=1_G$ then
$\nu$ is called a norm. If $\nu(ghg^{-1})=\nu(h)$ then
$\nu$ is called conjugation invariant.

\begin{remark*}
If $G$ is an abelian group then a norm is often required to be homogeneous.
That is, $\nu(ng)=|n|\nu(g)$ for all $g\in G$ and all integers $n\in \B Z$.
We do not make this requirement here.
\end{remark*}

The stabilization of $\nu$ is defined by
$$
s\nu(g):=\lim_{k\to \infty}\frac{\nu\left(g^k\right)}{k}.
$$
The stabilization of a norm does not have to be
a pseudonorm. Both the nontriviality and the triangle
inequality can be violated. If $G$ is abelian, however,
then the stabilization of a norm is a pseudonorm.
A norm $\nu$ is called {\em stably unbounded} if there exists
$g\in G$ such that $s\nu(g)\neq 0$.

\subsection{The biinvariant word norm}
Let $G$ be a normally finitely generated group. This
means that there exists a finite symmetric set
$S\subset G$ such that its normal closure $\overline S$
generates $G$. We also say that $S$ normally generates $G$.
The associated word norm is defined by
$$
\|g\|:=\min\{k\in \B N\,|\, g=s_1...s_k,
\text{ where } s_i\in \overline{S}\}.
$$
This norm is, by definition, conjugation invariant and hence
the induced metric, defined by $d_S(g,h):=\|gh^{-1}\|$ is
biinvariant. The standard argument shows that
any homomorphism $G\to H$ is Lipschitz with respect
to $d_S$ and any biinvariant metric on $H$.
In particular, the Lipschitz
class of this metric does not depend on the choice of
a finite set normally generating $G$.

\subsection{The commutator length}\label{SS:cl}
Let $g\in [G,G]$. Its commutator length is defined by
$$
\OP{cl}(g)
:=\min\{k\in \B N\,|\,g=[a_1,b_1]\dots [a_k,b_k],\,\text{where } a_i,b_i\in G\}.
$$
This quantity has been extensively studied and we refer the reader
to Calegari's book \cite{MR2527432} for more information. The stable commutator
length of an element $g$ is denoted by $\OP{scl}(g)$.
Let us explain that for a braid group $\B B_n$
the vanishing of the stable commutator
length is equivalent to the vanishing of the stabilization
of the biinvariant word norm.

It is known that braid groups satisfy the bq-dichotomy,
see \cite[Theorem 5.E]{bgkm}. This means that
for every element $\alpha \in \B B_n$ the cyclic subgroup
$\langle \alpha \rangle$ is either biinvariantly bounded or
there exists a homogeneous quasimorphism $q\colon G\to\B R$ such that
$q(\alpha)\neq 0$.
If $\alpha \in [\B B_n,\B B_n]$ then the bq-dichotomy implies that
the stable commutator length is trivial if and only if $\alpha$
generates a bounded cyclic subgroup. Consequently
the stable commutator length of $\alpha$ is trivial if and only
if the stable biinvariant word norm of $\alpha$ is trivial.

On the other hand, the commutator subgroup of the
infinite braid group contains
undistorted elements, according to Corollary \ref{C:binfty}. Since
the stable commutator length is trivial on $\B B_{\infty}$ these
elements are not detected by a quasimorphism. Thus the infinite braid
group does not satisfy the bq-dichotomy.

\subsection{The four ball genus norm on the concordance group}
Let $\OP{Conc(\B S^3)}$ denote the abelian group of smooth concordance classes
of knots in $\B S^3$. Two knots $K_0,K_1\in \B S^3=\partial \B B^4$
are {\em concordant} if there exists a smooth embedding
$c\colon\B S^1\times [0,1]\to \B B^4$ such that
$c(\B S^1\times \{0\})=K_0$ and $c(\B S^1\times \{1\})=K_1$.
The knot is called {\em slice} if it is concordant to the unknot.
The addition in $\OP{Conc}(\B S^3)$ is defined by the connected
sum of knots. The inverse of an element $[K]\in \OP{Conc}(\B S^3)$ is
represented by the knot $-K^*$, where $-K^*$ denotes the mirror image of the
knot $K$ with the reversed orientation. This group is equipped with a norm
defined by the four ball genus. More precisely,
$\OP{g}_4[K]$ is the minimal genus of an embedded surface in $\B B^4$
bounded by the knot $K$. We will call it the
{\em four ball genus norm}.
Its stabilization is denoted by
$\OP{sg_4}[K]$. For more information about the group $\OP{Conc(\B S^3)}$ see \cite{MR2179265}.

\subsection{The knot closure of a braid and the definition of $\Psi_n$}
\label{SS:Psin}
Let $\B B_n$ be the braid group on $n$-strings and let
$\sigma_1,\ldots,\sigma_{n-1}$ denote the standard
Artin generators.
We are interested in closures of braids in $\B S^3$.
In general, the closure of a braid has many components. In this section we describe the procedure
which produces a knot from a braid.
The closure of a braid $\alpha$ will be denoted either
by $\OP{closure}(\alpha)$ or by $\widehat{\alpha}$.

Let us introduce some notation.
Let $(\phantom{a})\colon \B B_n\to \Sigma_n$ be the projection onto the
symmetric group on $n$ letters. The elements $(\sigma_i)$ are then
the transpositions $(i,i+1)$. Let $\iota_n\colon \B B_n\to \B B_{n+1}$
denote the inclusion onto the first $n$ strands.

Let $\B B_n^K$ denote the set of braids on $n$ strands consisting
of braids whose closures are knots. It is a conjugation
invariant set and it is the preimage of the set of
the longest cycles with respect to the projection to
the symmetric group. We define a projection
$\pi_n\colon \B B_n\to \B B_n^K$ as follows.

Given a braid $\alpha \in \B B_n$
we construct a braid $\sigma_{(\alpha)}$
depending only on the permutation $(\alpha)\in \Sigma_n$
induced by $\alpha$ such that the composed braid
$\alpha\sigma_{(\alpha)}$ induces a longest cycle.
More precisely, let
$$
(\alpha) = (a_{1,1}\ldots a_{1,n_1})(a_{2,1}\ldots a_{2,n_2})\dots
(a_{k,1}\ldots a_{k,n_k})
$$
be presented as a cycle ordered lexicographically. We also require that
$\sum_{i=1}^k n_i=~n$, that is, we list cycles of length one.
Then we define
$$
\sigma_{(\alpha)} := \sigma_{a_{2,1}-1}\sigma_{a_{3,1}-1}\dots \sigma_{a_{k,1}-1}.
$$
The permutation induced by $\alpha\sigma_{(\alpha)}$ is then
the longest cycle obtained inductively by inserting the
second cycle to the first one,
the third cycle into the resulting cycle and so on.

A geometric description of the procedure goes as follows.
Consider the closure of $\alpha$ and color the component
containing the first strand red. Move to the left and if
the $i$-th strand is not red then multiply $\alpha$ by
$\sigma_{i-1}$, extend the coloring and continue the
procedure. The following properties are clear directly
from the construction:
\begin{itemize}
\item
The closure of $\alpha \sigma_{(\alpha)}$ is a knot.
\item
The braid $\sigma_{(\alpha)}$ is a product of $C_{\alpha}-1$
transpositions, where $C_{\alpha}$ is the number of components
of the closure of $\alpha$.
\item
The closure of $\sigma_{(\alpha)}$ is a trivial link.
\item
If $\alpha \in \B B_n^K$ then $\sigma_{(\alpha)}$ is the
identity; in other words $\alpha\mapsto \alpha\sigma_{(\alpha)}$
defines a projection $\pi_n\colon \B B_n\to \B B_n^K$.
\end{itemize}

Next we
define a map $\Psi_n\colon \B B_n\to \OP{Conc}(\B S^3)$ to be
the composition of the projection $\pi_n$ followed by the
closure of a braid and taking the concordance class:
$$
\Psi_n(\alpha) = [\OP{closure}(\pi_n(\alpha))]=
[\widehat{\alpha\sigma_{(\alpha)}}].
$$
Now the proof of the following observation is
straightforward.

\begin{proposition}\label{P:inclusion}
Let $\iota_n \colon \B B_n\to \B B_{n+1}$ be the inclusion onto
the first $n$ strings. The following diagram is commutative
$$
\xymatrix
{
\B B_n     \ar[r]^{\pi_n} \ar[d]^{\iota_n} &
\B B_n^K \ar[drr]^{\OP{closure}}\ar[d]^{\iota_n(-) \sigma_n} & \\
\B B_{n+1} \ar[r]^{\pi_{n+1}} & \B B_{n+1}^K \ar[rr]^{\OP{closure}} &
& {\bf KNOTS}\ar[r]^{[\phantom{a}]} & \OP{Conc}(\B S^3)\\
}
$$
Consequently, $\Psi_{n+1}\circ \iota_n = \Psi_n$ and
the map $\Psi_{\infty}\colon \B B_{\infty}\to \OP{Conc}(\B S^3)$
is well defined and surjective.\qed
\end{proposition}

\begin{remark*}
The restriction of the map $\Psi_{\infty}$ to the commutator
subgroup $[\B B_{\infty},\B B_{\infty}]$ is surjective.
To see this let $K=\Psi_n(\alpha)$ where $\alpha\in \B B_n^K$.
Suppose that $\OP{Ab}(\alpha)=k$, where $\OP{Ab}\colon \B B_n\to \B Z$
is the abelianisation homomorphism. Observe that the closure of
the braid $\iota(\alpha)\sigma_n^{-1}\ldots\sigma_{n+k-1}^{-1}$ is equal to $K$ and
that this braid belongs to the commutator subgroup
$[\B B_{n+k},\B B_{n+k}]$.
\end{remark*}

\section{Proofs }\label{S:proofs}

\subsection{General facts about quasihomomorphisms}
The following lemma will be used in the proof of Corollary \ref{C:lipschitz-bi}.
Observe that the inequalities in the first part have a particularly simple form
if $\psi(\OP{1}_G)=~0$. For $\alpha,\beta\in G$ we denote by
$\alpha^{\beta}:=\beta\alpha\beta^{-1}$.
\begin{lemma}\label{L:qm}
Let $A$ be an abelian group equipped with a pseudonorm $\nu$ and let
$\psi\colon G\to A$ be a quasihomomorphism.
\begin{enumerate}
\item
The following inequalities hold for every $\alpha,\beta\in G$:
\begin{itemize}
\item
$\nu(\psi(\beta)+\psi(\beta^{-1}))\leq \nu(\psi(1_G))+D_{\psi}$
\item
$\nu(\psi(\alpha^{\beta})-\psi(\alpha))\leq \nu(\psi(\OP{1}_G)) + 3D_{\psi}$
\item
$\nu(\psi([\alpha,\beta]))\leq 2\nu(\psi(\OP{1}_G))+5 D_{\psi}$.
\end{itemize}
\item
If $\psi$ is
bounded on a set $S$ normally generating $G$ then it is Lipschitz
with respect to the biinvariant word metric on $G$.
In particular, $\psi$ is Lipschitz if $G$ is normally finitely generated.
\item
The restriction of $\psi$ to the commutator subgroup $[G,G]$
is Lipschitz with respect to the commutator length.
\end{enumerate}
\end{lemma}

\begin{proof}
\begin{enumerate}[leftmargin=*]
\item
All inequalities follow directly from the quasihomomorphism property.
\item
Let $\alpha=s_1^{\beta_1}\ldots s_k^{\beta_k}$, where $s_i\in S$.
According to the hypothesis $\psi $ is bounded on $S$. It follows
from the previous part that $\psi $ is bounded, say by $C\geq 0$,
on the normal closure of $S$. We have
\begin{eqnarray*}
\nu(\psi(\alpha))&=&\nu\left(\psi\left(s_1^{\beta_1}\ldots s_{k}^{\beta_{k}}\right)\right)\\
&\leq& \nu\left (\sum_{i=1}^k \psi(s_i^{\beta_i})\right ) + (k-1)D_{\psi}\\
&\leq& \sum_{i=1}^k \nu \left( \psi(s_i^{\beta_i})\right ) + (k-1)D_{\psi}\\
&\leq& (C+D_{\psi})k,
\end{eqnarray*}
and the statement follows.
\item
The last inequality of item $(1)$ shows that $\psi$ is bounded on commutators.
This implies that if $\alpha\in [G,G]$ is a product of $k$ commutators then
$$
\nu(\psi(\alpha)) \leq k(2\nu(\psi(\OP{1_G}))+5D_{\psi}) +
(k-1)D_{\psi} \leq (2\psi(\OP{1_G})+6D_{\psi})k.
$$
Thus the Lipschitz constant of the restriction of $\psi$ to the commutator
subgroup with respect to the commutator length is bounded by $2\psi(\OP{1_G})+6D_{\psi}$.
\end{enumerate}
\end{proof}

\subsection{Proof of the first results and basic consequences}
Recall that given two knots $K$ and $K'$ we denoted by $-K$ the knot $K$ with
the reversed orientation, by $K^*$ the knot which is the mirror image of the
knot $K$, and by $K\#K'$ the connected sum of $K$ and $K'$. In \cite[Lemma
2.7]{MR2851716} the first author proved the following lemma (we reproduce the
proof for completeness).

\begin{lemma}\label{L:bordism}
Let $\alpha,\beta\in \B B_n$. There exists a smooth
bordism $\Sigma\to \B B^4$ between the knots
$$
\widehat{\alpha\sigma_{(\alpha)}}\#\widehat{\beta\sigma_{(\beta)}}\#
-(\widehat{\alpha\beta\sigma_{(\alpha\beta)}})^*\quad \text{ and }\quad
\widehat{\alpha\beta\sigma_{(\alpha\beta)}}\#
-(\widehat{\alpha\beta\sigma_{(\alpha\beta)}})^*
$$
such that $\chi(\Sigma)\geq-6n$.
\end{lemma}
\begin{proof}
The proof relies on the observation that if a link $L$ is obtained
from a link $L'$ by the operation presented in Figure \ref{F:saddle-move}
then there is an oriented bordism between $L$ and $L'$ of Euler characteristic
equal to~$-1$.
\begin{figure}[htb]
\centerline{\includegraphics[height=1in]{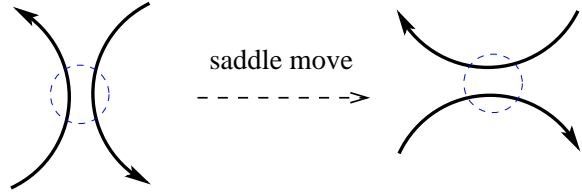}}
\caption{\label{F:saddle-move}
Saddle move which results in a cobordism
of Euler characteristic $-1$ between links $L$ and $L'$.}
\end{figure}

Applying this argument inductively we get that the there is a bordism
between $\widehat{\alpha\beta}$ and $\widehat{\alpha}\sqcup \widehat{\beta}$,
where $\alpha,\beta\in \B B_n$ and that the Euler characteristic
of this bordism is equal to~$-n$.
In our situation we obtain the following sequence of bordisms
(the number over an arrow is an upper bound on the number of one-handles
attached to the previous bordism):
\begin{align*}
\widehat{\alpha \beta \sigma_{(\alpha\beta)}} &\stackrel{2n}{\xrightarrow{\hspace*{1cm}}}
\widehat{\alpha} \sqcup \widehat{\beta} \sqcup \widehat{\sigma_{(\alpha\beta)}} \\
&\stackrel{2n-1}{\xrightarrow{\hspace*{1cm}}} \widehat{\alpha} \sqcup \widehat{\beta}
\sqcup \widehat{\sigma_{(\alpha)}} \sqcup \widehat{\sigma_{(\beta)}} \\
&\stackrel{2n}{\xrightarrow{\hspace*{1cm}}} \widehat{\alpha\sigma_{(\alpha)}}
\sqcup \widehat{\beta\sigma_{(\beta)}}\\
&\stackrel{1}{\xrightarrow{\hspace*{1cm}}}
\widehat{\alpha\sigma_{(\alpha)}}\# \widehat{\beta\sigma_{(\beta)}}
\end{align*}
The number of handles in the second bordism follows from an observation
that the closure $\widehat{\gamma_{\sigma_{(\gamma)}}}$ is a trivial
link with at most $n$ components for any $\gamma\in \B B_n$.
It follows that there is a bordism between
$$
\widehat{\alpha\beta\sigma_{(\alpha\beta)}}\#
-(\widehat{\alpha\beta\sigma_{(\alpha\beta)}})^*
\stackrel{6n}{\xrightarrow{\hspace*{1cm}}}
\widehat{\alpha\sigma_{(\alpha)}}\#\widehat{\beta\sigma_{(\beta)}}\#
-(\widehat{\alpha\beta\sigma_{(\alpha\beta)}})^*
$$
which is the cylinder with at most $6n$ handles attached which
implies the statement.
\end{proof}

Since the second knot in the above lemma is slice, we obtain
that the four ball genus of the first knot is bounded by $3n+1$.

\begin{proof}[Proof of Theorem \ref{T:main}]
Let $\alpha,\beta\in \B B_n$.
\begin{eqnarray*}
\OP{g}_4\left(\Psi_n(\alpha)+\Psi_n(\beta)-\Psi_n(\alpha\beta)\right)&=&
\OP{g}_4\left[\widehat{\alpha\sigma_{(\alpha)}}\#\widehat{\beta\sigma_{(\beta)}}\#
-(\widehat{\alpha\beta\sigma_{(\alpha\beta)}})^*\right]\\
&\leq& 3n+1
\end{eqnarray*}
This proves that $\Psi_n\colon \B B_n\to \OP{Conc}(\B S^3)$
is a quasihomomorphism with defect bounded by $3n+1$.
\end{proof}

Let us specify the general inequalities from Lemma \ref{L:qm}
to our situation.
\begin{corollary}\label{C:estimates}
The quasihomomorphism  $\Psi_n\colon \B B_n\to \OP{Conc}(\B S^3)$
satisfies the following inequalities for every $\alpha,\beta\in \B B_n$.
\begin{itemize}
\item
$\OP{g_4}(\Psi_n(\alpha) + \Psi_n(\alpha^{-1})\leq D_{\Psi_n} \leq 3n+1$.
\item
$\OP{g_4}(\Psi_n(\alpha^{\beta}-\Psi_n(\alpha))\leq 3D_{\Psi_n}\leq 9n+3$.
\item
$\OP{g_4}(\Psi_n([\alpha,\beta]))\leq 5D_{\Psi_n}\leq 15n + 5$.
\end{itemize}
\end{corollary}
\begin{proof}
Since the closure of $\sigma_1\dots \sigma_{n-1}$ is the unknot we
get that $\Psi_n(\OP{1}_{\B B_n})$ is equal to the trivial concordance class
 and hence
$\OP{g_4}(\Psi_n(\OP{1}_{\B B_n}))=0$. Consequently, the above
inequalities follow directly from Lemma \ref{L:qm} and
Theorem \ref{T:main}.
\end{proof}

\begin{proposition}\label{P:infty}
The sequence of the defects of the quasihomomorphisms
$\Psi_n\colon \B B_n\to \OP{Conc}(\B S^3)$
is unbounded:
$$
\limsup_{n\to \infty}D_{\Psi_n}=\infty.
$$
\end{proposition}
\begin{proof}
If the defects were uniformly bounded then the map $\Psi_{\infty}$ would be a
quasihomomorphism. This would imply, according to Kotschick \cite{MR2509718},
that the composition $\B B_3\to \B B_{\infty} \to \OP{Conc}(\B S^3)\to \B R$,
where the last map is given by the signature link invariant, is a bounded
distance from a homomorphism.  However, it is known that this composition is a
quasimorphism, whose homogenization which does not vanish on the commutator
subgroup $[\B B_3,\B B_3]$. This follows, for example, from the fact that the
non-trivial homogeneous signature quasimorphism $\overline{\OP{sign}}_3$ on $\B B_3$
defined in \cite{MR2104597} is not a homomorphism. Indeed, if it is a
non-trivial homomorphism, then its value on the braid $\eta_{3,3}$ must be
equal to twice its value on the braid $\eta_{2,3}$, where the braids
$\eta_{2,3}$ and $\eta_{3,3}$ are shown in Figure \ref{fig:braids-eta-i-n}.
However, in \cite{MR2104597} Gambaudo-Ghys showed that
$\overline{\OP{sign}}_3(\eta_{2,3})=\overline{\OP{sign}}_3(\eta_{2,3})=2$.
\end{proof}

\begin{remark*} The above proposition does not exclude the possibility that
$\Psi_{\infty}\colon \B B_{\infty}\to \OP{Conc}(\B S^3)$ is
Lipschitz. It can't be Lipschitz, however, with respect to the commutator
length because the latter is bounded by two on the infinite braid group
\cite[Theorem 2.2]{MR2509711}.
\end{remark*}

\subsection*{Proof of Theorem \ref{T:lipschitz}}
The main ingredient of the proof is the following observation.
\begin{lemma}\label{L:r1}
Let $\alpha,\beta\in \B B_n$. Suppose that
$\alpha=\sigma_{i_1}^{\pm 1}\dots \sigma_{i_m}^{\pm 1}\in \B B_n$ and
$\beta=\sigma_{i_1}^{\pm 1}\dots \sigma_{i_{k-1}}^{\pm 1}
\sigma_{i_{k+1}}^{\pm }\dots\sigma_{i_m}^{\pm 1}$.
That is, $\beta$ is obtained from $\alpha$ by removing one crossing.
Then there is a smooth bordism $\Sigma\to\B B^4$ from
the closure $\widehat{\alpha}$ to the closure $\widehat{\beta}$ whose Euler
characteristic is equal to $-1$.
\end{lemma}

\begin{proof}
It is enough to argue locally at a neighborhood of a crossing.
The proof for removal of $\sigma_i$ is presented in Figure
\ref{F:saddle-cobordism}. The proof for removal of $\sigma_i^{-1}$ is
analogous.
\begin{figure}[htb]
\centerline{\includegraphics[height=2in]{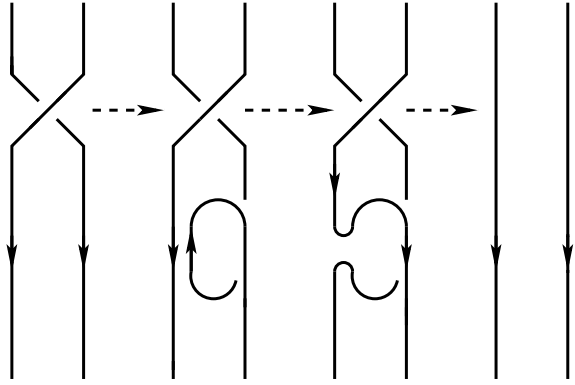}}
\caption{\label{F:saddle-cobordism} Local moves.}
\end{figure}

The first step is to change the braid after the crossing by the first
Reidemeister move (appropriately chosen)
so the neighboring strings go in opposite directions.
Next we perform a saddle move.
Then we apply the second Reidemeister move which results in the
braid $\beta$. We obtain
that there is a bordism from the closure of $\alpha$ to the
closure of $\beta$ whose Euler characteristic is equal to $-1$.
\end{proof}

\begin{proof}[Proof of Theorem \ref{T:lipschitz}]
Let $\alpha \in \B B_n$. Suppose that $\|\alpha\|=k$ which means
that
$$
\alpha = \beta_1\sigma_1^{\pm 1}\beta_1^{-1}\dots \beta_k\sigma_1^{\pm 1}\beta_k^{-1}
$$
for some $\beta_i\in \B B_n$. The knot $\Psi_n(\alpha)$ is the
closure of the braid
$$
\beta_1\sigma_1^{\pm 1}\beta_1^{-1}\dots \beta_k\sigma_1^{\pm 1}\beta_k^{-1}\sigma_{(\alpha)}.
$$
By applying Lemma \ref{L:r1} $k$ times we remove the crossings corresponding
to $\sigma_1^{\pm 1}$ in the above presentation and we obtain that the above closure
is bordant to the closure of $\sigma_{(\alpha)}$ via a bordism whose Euler
characteristic is equal to $-k$.

The closure of $\sigma_{(\alpha)}$ is a trivial link and we  cap off all
its components. This increases the Euler characteristic by the number of
components which is at least one and
yields a surface of genus at most $k/2$
in $\B B^4$ bounded by $\Psi_n(\alpha)$. We get that
$$
\OP{g_4}(\Psi_n(\alpha))\leq \frac{k}{2} = \frac{1}{2}\|\alpha\|
$$
as claimed.
\end{proof}

\subsection{Proofs of the corollaries}

\begin{proof}[Proof of Corollary \ref{C:compositions}]
Let $\alpha,\beta \in \B B_n$ and let $\phi\colon \OP{Conc}(\B S^3)\to \B R$
be a quasimorphism which is Lipschitz with respect to the four ball genus norm.
Let $E = \Psi_n(\alpha\beta) - \Psi_n(\alpha) - \Psi_n(\beta)$.
It follows from the quasihomomorphism property of $\Psi_n$ that
$\OP{g_4}(E)\leq D_{\Psi_n}$.
\begin{align*}
|\phi(\Psi_n(\alpha)) &- \phi(\Psi_n(\alpha\beta)) + \phi(\Psi_n(\beta))|\\
&\leq
|\phi(\Psi_n(\alpha)) - \phi(\Psi_n(\alpha)+\Psi_n(\beta)+E) + \phi(\Psi_n(\beta))|\\
&\leq
2D_{\phi} + |\phi(E)| \leq 2D_{\phi} + C_{\phi} D_{\Psi_n},
\end{align*}
where $C_{\phi}$ is the Lipschitz constant of $\phi$.
\end{proof}

\begin{proof}[Proof of Corollary \ref{C:lipschitz-bi}]
First we prove that the quasihomomorphism $\Psi_n$ is Lipschitz with
respect to the biinvariant word metric.
Let $\alpha \in \B B_n$ be an element of the biinvariant
word norm $\|\alpha\|$ equal to $k$. This means that
$\alpha =(\sigma_1^{\pm 1})^{p_1}\dots (\sigma_1^{\pm})^{p_k}$
for some $p_i\in \B B_n$.
\begin{align*}
\OP{g_4}(\Psi_n(\alpha))
&= \OP{g_4}(\Psi_n(\sigma_1^{\pm 1})^{p_1}\dots (\sigma_1^{\pm})^{p_k})\\
&\leq \sum_{i=1}^k \OP{g_4}(\Psi_n((\sigma_1^{\pm 1})^{p_i})) + (k-1)D_{\Psi_n}\\
&\leq D_{\Psi_n}(k-1) \\
&\leq D_{\Psi_n}\|\alpha\|\leq (3n+1)\|\alpha\|.
\end{align*}
The first inequality follows from the quasihomomorphism
property and
the second inequality is a consequence of the fact that the closure
of $\sigma_1^{\pm 1}$ is the unknot.

Let us now consider the restriction of $\Psi_n$ to the commutator
subgroup $[\B B_n,\B B_n]$ equipped with the commutator length.
Since
$$\OP{g_4}(\Psi_n[\alpha,\beta])\leq 5D_{\Psi_n} \leq 15n+5,$$
due to Corollary \ref{C:estimates}, we get that if $\alpha\in [\B B_n,\B B_n]$ then
$$
\OP{g_4}(\Psi_n(\alpha))\leq 6D_{\Psi_n}\OP{cl}(\alpha) \leq (18n+6)\OP{cl}(\alpha).
$$
\end{proof}

\begin{proof}[Proof of Corollary \ref{C:scl}]
Let $\alpha\in [\B B_n,\B B_n]$.
It follows from the quasihomomorphism property that
\begin{equation*}
\OP{g}_4\left(\Psi_n\left(\alpha^{k}\right)-k\Psi_n(\alpha)\right)
\leq (k-1)D_{\Psi_n}.
\end{equation*}
By dividing by $k$ and taking the limit we obtain that
\begin{equation}\label{eq:qm-g-st}
\left|\limsup_{k\to \infty}\frac{\OP{g}_4(\Psi_n(\alpha^k))}{k}-
\OP{sg}_4(\Psi_n(\alpha))\right|
\leq D_{\Psi_n},
\end{equation}
Since $\Psi_n$ is Lipschitz with respect to the commutator length
(the Lipschitz constant is computed in the proof of Corollary
\ref{C:lipschitz-bi}), we have that
$$
\limsup_{k\to \infty}\frac{\OP{g}_4(\Psi_n(\alpha^k))}{k}\leq
\lim_{k\to \infty}\frac{6D_{\Psi_n}\OP{cl}(\alpha^k)}{k}
=6D_{\Psi_n}\OP{scl}(\alpha).
$$
It then follows from Equation (\ref{eq:qm-g-st}) that
$$
\OP{sg_4}(\Psi_n(\alpha))\leq 6D_{\Psi_n}\OP{scl}(\alpha) + D_{\Psi_n}
$$
as claimed.
\end{proof}

\begin{proof}[Proof of Corollary \ref{C:bounded}]
Since the map $\Psi_n$ is Lipschitz by Corollary \ref{C:lipschitz-bi}
we obtain that
$$
\OP{g}_4(\Psi_n(\alpha^k))\leq (3n+1)\|\alpha^k\|.
$$
The stable commutator length of $\alpha$ is trivial which
implies, according to the bq-dichotomy (see Section \ref{SS:cl}),
that the cyclic subgroup generated by $\alpha$ is bounded.
As a consequence we get the uniform bound on the four
ball genus of $\Psi_n(\alpha^k)$.
\end{proof}

\begin{proof}[Proof of Corollary \ref{C:unboundedsg}]
Recall that $\varphi\colon \OP{Conc}(\B S^3)\to \B R$ is
a Lipschitz quasimorphism with the Lipschitz constant $C_{\varphi}$.
Let $\alpha\in \B B_n$ and $p\in \B N$.
\begin{align*}
\OP{sg_4}(\Psi_n(\alpha^p))&=\limsup_k \frac{\OP{g_4}(k\,\Psi_n(\alpha^p))}{k}\\
&\geq \limsup_k\frac{\OP{g_4}(\Psi_n(\alpha^{kp})) - (k-1)D_{\Psi_n}}{k}\\
&\geq \frac{1}{C_{\varphi}}
\limsup_k\frac{\left|(\varphi\circ\Psi_n)(\alpha^{kp})\right|}{k} - D_{\Psi_n}\\
&\geq  \frac{1}{C_{\varphi}}
\left|(\overline{\varphi\circ\Psi_n})(\alpha^{p})\right| - D_{\Psi_n}\\
&= \frac{|(\overline{\varphi\circ\Psi_n})(\alpha)|}{C_{\varphi}}\cdot p - D_{\Psi_n}
\end{align*}
The first inequality follows from the quasihomomorphism property,
the second is the Lipschitz property of $\varphi$ and the
third is the definition of the homogenization of a quasimorphism.
\end{proof}

\begin{proof}[Proof of Corollary \ref{C:binfty}]
Let $\alpha \in [\B B_{\infty},\B B_{\infty}]$. Since the abelianisations
of all braid groups are isomorphic, we have that
$\alpha\in [\B B_n,\B B_n]$ for some
$n\in \B N$. Let $\varphi\colon \OP{Conc}(\B S^3)\to \B R$ be a Lipschitz
quasimorphism such that $\varphi(\Psi_{n}(\alpha))\neq~0$.
We have the following inequalities:
$$
\frac{|(\overline{\varphi\circ \Psi_n}(\alpha)|}{C_{\varphi}}\cdot p - D_{\Psi_n}
\leq \OP{sg_4}(\Psi_n(\alpha^p))\leq \OP{g_4}(\Psi_n(\alpha^p))=\OP{g_4}(\Psi_\infty(\alpha^p))\leq
\frac{1}{2}\|\alpha^p\|.
$$
The first inequality follows from Corollary \ref{C:unboundedsg}, the second
one is obvious, the equality follows from Proposition \ref{P:inclusion} and the
last one follows from Theorem \ref{T:lipschitz}.  By dividing by $p$ and
passing to the limit with $p\to \infty$ we
obtain that $\lim_{p\to \infty}\frac{\|\alpha^p\|}{p}>0$ as claimed.
\end{proof}

\subsection{Strong displaceability of braids}\label{SS:disp}

Let $m\in \B N$ be a natural number. A subgroup $H\subset G$
is called strongly $m$-displaceable if there exists $g\in G$
such that the conjugate subgroups $g^iHg^{-i}$ and $g^{j}Hg^{-j}$
commute for $0\leq i<j\leq m$.

\begin{proposition}\label{P:disp}
For every natural numbers $m,n\in \B N$ the braid group
$\B B_n$ is strongly $m$-displaceable in $\B B_{\infty}$
and the commutator subgroup  $[\B B_n,\B B_n]$ is strongly
$m$-displaceable in $[\B B_{\infty},\B B_{\infty}]$.
\end{proposition}

\begin{proof}
For every $n\in \B N$ we define an {\em argyle braid} $A_{n,i}$
to be
$$
A_{n,i}:=
\prod_{k=1}^n\prod_{j=1}^n \sigma_{in-k+j}
$$
This is a braid that swaps the $i$-th $n$ strings with
the $(i+1)$-th $n$ strings, it is a product of $n^2$
standard generators and it looks like an argyle pattern.
For example, $A_{1,i}=\sigma_i$.
If $n$ is even then by making the pattern alternating
we define a {\em commutator argyle braid} by
$$
A'_{n,i}:=
\prod_{k=1}^n\prod_{j=1}^n \sigma_{in-k+j}^ {(-1)^{j}}
$$
An example of $A'_{4,1}$ drawn in Figure \ref{fig:argyle-braid}.
\begin{figure}[htb]
\centerline{\includegraphics[height=1.9in]{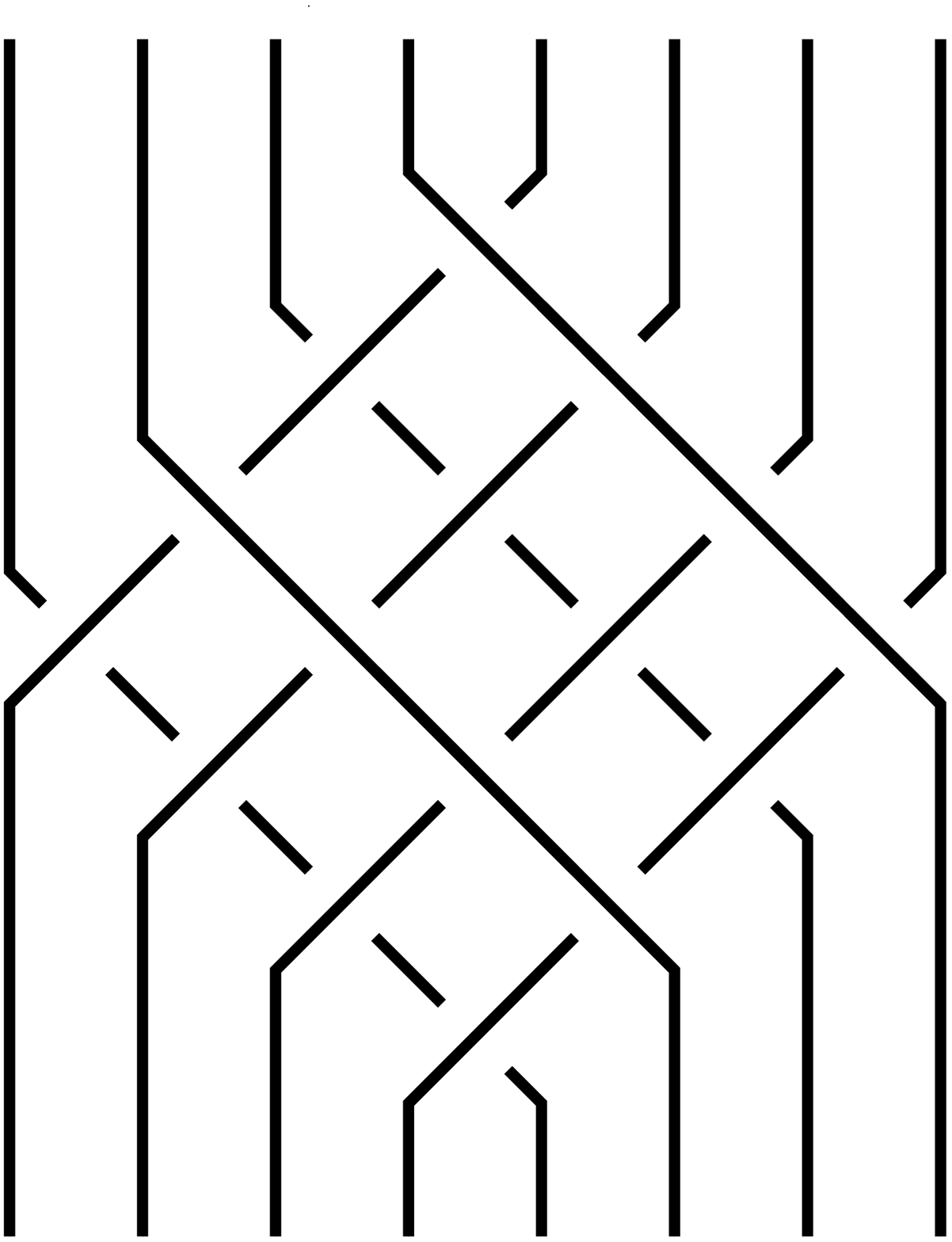}}
\caption{\label{fig:argyle-braid} The commutator argyle braid $A'_{4,1}$.}
\end{figure}

Let $\B B_n\subset \B B_{\infty}$, where $n$ is a positive and even
natural number. For any $m\in \B N$ we define
$$
\Delta:= A'_{n,1}A'_{n,2}\dots A'_{n,m-1}\in [\B B_{mn},\B B_{mn}].
$$
The braid $\Delta$ is presented in Figure \ref{fig:argyle-delta-braid} in
which each line represents $n$ strings and
each crossing is the appropriate argyle braid~$A'_{n,i}$.
\begin{figure}[htb]
\centerline{\includegraphics[height=2.3in]{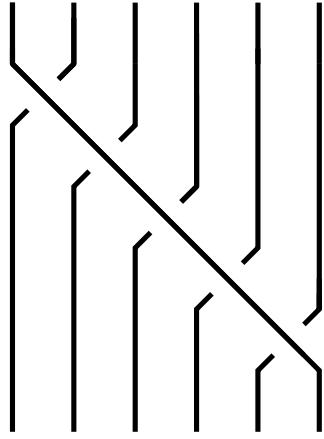}}
\caption{\label{fig:argyle-delta-braid} The braid $\Delta$ for $m=6$.}
\end{figure}

It shows that the group $\B B_n$ is strongly $m$-displaceable in $\B B_{mn}$ by the braid $\Delta$
and that the commutator subgroup $[\B B_n,\B B_n]$ is strongly $m$-displaceable
in $[\B B_{mn},\B B_{mn}]$. Since the strong $m$-displaceability of
$\B B_n$ implies the strong $m$-displaceability of $\B B_{n-1}$,
this observation implies the statement and finishes the proof.
\end{proof}

\section{Examples}\label{S:examples}

\subsection{Bounded cyclic subgroup of braid groups}\label{SS:boundedcyclic}

The following lemma was proved in \cite{bgkm} but, since its proof
is short, we present it for completeness.
\begin{lemma}\label{L:boundedbraid}
Let $G$ be a normally finitely generated group and let $\alpha,\Delta\in G$.
If $\alpha$ commutes with the conjugate $\alpha^{\Delta}$ then the
cyclic subgroup generated by $[\alpha,\Delta]$ is bounded with
respect to the biinvariant word metric on $G$.
In particular, if $\alpha$ is conjugate to its inverse then
$\alpha$ generates a bonded cyclic subgroup.
\end{lemma}
\begin{proof}
An induction argument yields the following equality
$$
[\alpha,\Delta]^n=[\alpha^n,\Delta].
$$
This implies that
$
\left\|[\alpha,\Delta]^n\right\| =
\left\|[\alpha^n,\Delta]\right\| \leq 2\|\Delta\|
$
which finishes the general part. If $\alpha$ is conjugate
to its inverse via $\Delta$ then $\alpha^2=[\alpha,\Delta]$
and the second statement follows. Moreover,
the braid $\|\alpha^n\|$ is bounded by $2\|\Delta\| + \|\alpha\|$
in this case.
\end{proof}

\begin{example}\label{E:s1s2}
Let $\sigma_1\sigma_2^{-1}\in \B B_3$ and let
$\Delta=\sigma_1\sigma_2\sigma_1$
be the Garside element (the half twist). Observe that
$$
(\sigma_1\sigma_2^{-1})^{\Delta}=\sigma_2\sigma_1^{-1},
$$
that is $\sigma_1\sigma_2^{-1}$ is conjugate to its inverse.
It then follows from Lemma \ref{L:boundedbraid} that the braid
$\sigma_1\sigma_2^{-1}$ generates a  cyclic
subgroup in $\B B_3$ bounded by
$8=2\|\Delta\|+\|\sigma_1\sigma_2^{-1}\|$.
\end{example}

\begin{example}\label{E:s1s3}
Let $\sigma_1\sigma_3^{-1}\in \B B_4$. It is conjugate
to its inverse via the braid which swaps the first two
and the last two strings. More precisely,
the conjugating braid is given in this case by
$\Delta = \sigma_2\sigma_1\sigma_3\sigma_2$.
Again, by the above lemma we
obtain that $\sigma_1\sigma_3^{-1}$ generates a
cyclic subgroup of $\B B_4$ bounded by
$10=2\|\Delta\|+\|\sigma_1\sigma_3^{-1}\|$.
\end{example}

\subsection{The notorious family $\Psi_3((\sigma_1\sigma_2^{-1})^n)$.}
\label{SS:notorious}
We saw in the previous section that the cyclic subgroup
generated by $\sigma_1\sigma_2^{-1}$ is bounded in $\B B_3$.
According to Corollary
\ref{C:lipschitz-bi}, Corollary \ref{C:bounded} and Example
\ref{E:s1s2} we get the following bound
on the four ball genus
$$
\OP{g_4}(\Psi_3((\sigma_1\sigma_2^{-1})^n))
\leq \frac{1}{2}\|(\sigma_1\sigma_2^{-1})^n\|
\leq 4.
$$
So we obtain a very simple example of a infinite
family of knots with uniformly bounded four ball
genus. However, this family has been notorious in
the sense that it remains an open problem whether
the induced family of concordance classes
$\Psi_3((\sigma_1\sigma_2^{-1})^n)$
for $n$ not divisible by three is
infinite. It is known that these concordance classes
are of order at most two so the question is whether
these knots are slice.

\subsection{The family $\Psi_4((\sigma_1\sigma_3^{-1})^n)$}
\label{SS:s1S3}
By the same argument as above we get that the four ball genus
of the knots $\Psi_4((\sigma_1\sigma_3^{-1})^n)$ is bounded
by $\frac{1}{2}\|(\sigma_1\sigma_3^{-1})^n)\|=5$. In this
case, however, we know that the set of induced concordance
classes is infinite. More precisely, we have the following
result.

\begin{proposition}\label{P:inf-family}
Let $\gamma=\sigma^2_1\sigma_3^{-2}\in\B B_4$. Then the set
$\{\Psi_4(\gamma^n)\}_{n=1}^\infty$ is infinite in $\OP{Conc}(\B S^3)$, and the
four ball genus of $\Psi_4(\gamma^n)$ is uniformly bounded. Moreover, there
exists an increasing sequence of natural numbers $\{n_i\}_{i=1}^\infty$ such
that the set $\{\Psi_4(\gamma^{n_i})\}_{i=1}^\infty$ generates $\B Z^\infty$ in
$\OP{Conc}(\B S^3)$.
\end{proposition}

\begin{proof}
By definition $\Psi_4(\gamma^n)$ equals to the concordance class of the closure
of the braid $\sigma^{2n}_1\sigma_3^{-2n}\sigma_1\sigma_2\sigma_3$.  For
$n\in\B N$ denote by $T_{2n-1}$ the knot obtained by taking a closure of the
braid $\sigma_1^{2n-1}\in\B B_2$.  It follows that the closure of the braid
$\sigma^{2n}_1\sigma_3^{-2n}\sigma_1\sigma_2\sigma_3$ equals to the knot
$T_{2n+1}\#(T_{2n-1})^*$.

The knot $T_{2n+1}$ is a $(2,2n+1)$ torus knot. Hence
$$
\Delta_{T_{2n+1}}(t)=\frac{(t^{4n+2}-1)(t-1)}{(t^2-1)(t^{2n+1}-1)}
=\sum_{i=0}^{2n}(-1)^it^{2n-i},
$$
where $\Delta_K(t)$ is the Alexander polynomial of a knot $K$. It follows that
$$
\OP{det}(T_{2n+1})=(2n+1), \quad \OP{det}((T_{2n-1})^*)=2n-1 \quad\rm{ and }
$$
$$
\OP{det}(T_{2n+1}\#(T_{2n-1})^*)=(2n+1)(2n-1),
$$
where the determinant of $K$ is defined to be $\OP{det}(K):=|\Delta_K(-1)|$.
Let $\{n_i\}_{i=1}^\infty$ be an increasing sequence of natural numbers such
that $p_i:=2n_i+1$ is a prime number.  It follows that for each $i\in\B N$ we
have
$$
\OP{det}(T_{p_i}\#(T_{p_i-2})^*)=p_i(p_i-2).
$$
Hence $\OP{det}(T_{p_i}\#(T_{p_i-2})^*)$ is not a square,
and for each $i>j\in \B N$
$$
\OP{det}(T_{p_i}\#(T_{p_i-2})^*\#-(T_{p_j})^*\#-(T_{p_j-2}))=p_i(p_i-2)p_j(p_j-2)
$$
is not a square. Since the determinant of a slice knot must be a square number,
the concordance classes of knots $(T_{p_i}\#(T_{p_i-2}))$ are pairwise distinct.
Hence the set $\{\Psi_4(\gamma^{n_i})\}_{i=1}^\infty$ is infinite in
$\OP{Conc}(\B S^3)$.

Let $\{p_i\}_{i=1}^\infty$ be a set of odd primes such that for each $i$
$$
p_i>2p_1\cdot\ldots\cdot p_{i-1}.
$$
Let $n_i:=p_1\cdot\ldots\cdot p_i$. In what follows we are going
to show that the set $\{\Psi_4(\gamma^{n_i})\}_{i=1}^\infty$ generates $\B
Z^\infty$ in $\OP{Conc}(\B S^3)$.

Recall that for each complex number $\omega\neq 1$, such that $|\omega|=1$,
there exists the $\omega$-signature (Levine-Tristram $\omega$-signature)
homomorphism $\OP{sign}_{\omega}\colon \OP{Conc}(\B S^3)\to \B Z$. For each odd prime $p$
denote by $\omega_p:=\exp\left(\frac{(p-1)\pi i}{p}\right)$. It follows from
\cite[Lemma 3.5]{MR0248854} that for each prime $p$ and each natural number $n$
we have
\begin{align*}
\OP{sign}_{\omega_p}(\Psi_4(\gamma^n))
&=\OP{sign}_{\omega_p}(T_{2n+1}\#(T_{2n-1})^*)\\
&=2-2\left(\left[\frac{n}{p}+\frac{1}{2p}\right]
-\left[\frac{n}{p}-\frac{1}{2p}\right]\right),
\end{align*}
where $[\cdot]$ denotes the integer part. Since $p_i>2p_1\cdot\ldots\cdot
p_{i-1}$ and $n_i:=p_1\cdot\ldots\cdot p_i$ we obtain
$$
\OP{sign}_{\omega_{p_{i+1}}}(\Psi_4(\gamma^{n_i}))=
\OP{sign}_{\omega_{p_{i+1}}}(T_{2n_i+1}\#(T_{2n_i-1})^*)=2\quad \rm{ and }
$$
$$
\OP{sign}_{\omega_{p_{i+1}}}(\Psi_4(\gamma^{n_j}))=
\OP{sign}_{\omega_{p_{i+1}}}(T_{2n_j+1}\#(T_{2n_j-1})^*)=
0\hspace{2mm} \textrm{ if }\hspace{2mm} i+1<j,
$$
and the proof follows.
\end{proof}

\begin{remark*}\rm
Note that the standard 3-dimensional genus of knots $\Psi_4(\gamma^n)$ goes to infinity when $n\to\infty$, since the genus of $T_{2n+1}=n$ and thus the genus of $\Psi_4(\gamma^n)$ equals to $2n-1$.
\end{remark*}

\subsection{Prime knots with unbounded stable genus}
\label{SS:unboundedsg}
This section provides details for Example \ref{E:signature}. We construct
a braid $\alpha\in \B B_n$ for $n\geq 3$ such that the four ball genus
of the knots $\Psi_n(\alpha^p)$ grows linearly with $p$.

Let $\OP{sign}(L)\in \B Z$ denote the signature invariant
of a link $L$. The restriction of the signature to knots
descends to a homomorphism
$$
\OP{sign}\colon \OP{Conc}(\B S^3)\to \B Z
$$
on the concordance group.
It is a well known fact due to Murasugi \cite{MR0171275} that
the following inequality
\begin{equation}\label{Eq:Murasugi}
|\OP{sign}(K)|\leq 2\OP{g}_4(K)
\end{equation}
holds for every knot $K$.
In other words the signature is Lipschitz with constant $C_{\OP{sign}}=2$.
It follows from Corollary \ref{C:compositions} that the
composition $\OP{sign}\circ \Psi_n\colon \B B_n\to \B R$
is a quasimorphism on the braid group.

In order to apply Corollary \ref{C:unboundedsg} we need to show that
the there exists a braid $\alpha\in \B B_n$ such that
the quasimorphism $\OP{sign}\circ \Psi_n$ is unbounded
on the cyclic subgroup generated by $\alpha$.

Let
$\eta_{i,n}:=\sigma_{i-1}\ldots\sigma_2\sigma_1^2\sigma_2\ldots\sigma_{i-1}\in \B B_n$,
be the braid presented in Figure \ref{fig:braids-eta-i-n}.
\begin{figure}[htb]
\centerline{\includegraphics[height=1.8in]{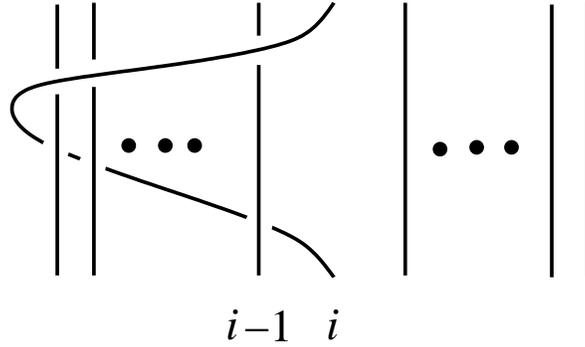}}
\caption{\label{fig:braids-eta-i-n} The braid $\eta_{i,n}$.}
\end{figure}

Let $\alpha=\eta_{2,n}^{-2}\eta_{3,n}\in \B B_n$. Observe that
$\alpha=\sigma_1^{-4}\sigma_2\sigma_1^2\sigma_2$ and hence
we get that $\alpha\in~[\B B_n,\B B_n]$. Notice moreover
that $\alpha$ is a pure braid and hence for
each integer $p$ we have $\sigma_{(\alpha^p)}=\delta$, where
$\delta:=\sigma_1\ldots\sigma_{n-1}$.

Let $\OP{sign}_n\colon \B B_n\to \B Z$ be a function defined by
$\OP{sign}_n(\beta)=\OP{sign}(\widehat{\beta})$. Gambaudo and Ghys showed in
\cite{MR2104597} that $\OP{sign}_n$ is a quasimorphism on
$B_n$ with a defect $D_{\OP{sign}_n}\leq n-1$. We denote by $\overline{\OP{sign}}_n$
the induced homogeneous quasimorphism. They also proved that
\begin{equation*}
\overline{\OP{sign}}_n(\eta_{i,n})=\left\{
\begin{array}{c}\begin{aligned}
&i,\hspace{10.5mm} \rm{if}\hspace{2mm} \textit{i}\hspace{2mm} \rm{is}\hspace{2mm} \rm{even},\\
&i-1,\hspace{4.4mm} \rm{if}\hspace{2mm} \textit{i}\hspace{2mm} \rm{is}\hspace{2mm} \rm{odd}.\\
 \end{aligned}
 \end{array}
 \right.
\end{equation*}
Since  the braids $\eta_{i,n}$ pairwise commute we have
$$\overline{\OP{sign}}_n(\alpha)
=\overline{\OP{sign}}_n(\eta_{2,n}^{-2})+\overline{\OP{sign}}_n(\eta_{3,n})=-2.$$

Since the closure $\widehat{\delta}$ is the unknot we get that
$$
|\OP{sign}_n(\alpha^p_n\delta)-\OP{sign}_n(\alpha^p_n)|\leq n-1,
$$
for every integer $p\in \B Z$. It follows from a general fact about the
homogenization of quasimorphisms that
$$
|\OP{sign}_n(\alpha^p_n)-\overline{\OP{sign}}_n(\alpha^p_n)|\leq D_{\OP{sign}_n}\leq n-1.
$$
By combining the two inequalities we obtain that
$$
|\OP{sign}_n(\alpha^p_n\delta)-\overline{\OP{sign}}_n(\alpha^p_n)|
=|\OP{sign}_n(\alpha^p_n\delta)+2p|\leq 2n-2,
$$
for every integer $p\in \B Z$. Recall that
$\OP{sign}_n(\alpha^p\delta)=(\OP{sign}\circ \Psi_n)(\alpha^p)$.
The above inequality then says that the restriction of the
quasimorphism  $\OP{sign}\circ \Psi_n$ to the cyclic subgroup
generated by $\alpha$ is within bounded distance from the
homogeneous quasimorphism $\overline{\OP{sign}_n}$ restricted
to the cyclic subgroup $\langle \alpha \rangle \subset \B B_n$.
This implies that $\OP{sign}\circ \Psi_n$ is unbounded on
the cyclic subgroup generated by $\alpha$ and, moreover,
its homogenization restricted to $\langle \alpha\rangle$
is equal to $\overline{\OP{sign}_n}$. It is then a consequence
of Corollary \ref{C:unboundedsg} that
$$
\OP{sg_4}(\Psi_n(\alpha^p))\geq
\frac{\left|(\overline{\OP{sign}\circ\Psi_n})(\alpha)\right|}{C_{\OP{sign}}}\cdot p - D_{\Psi_n}
\geq p - 3n +1.
$$

\bibliography{bibliography}
\bibliographystyle{acm}

\end{document}